\documentclass[reqno]{amsart}
\usepackage{amsbsy}
\usepackage{amsfonts}
\usepackage{amsmath}
\usepackage{amsthm}
\usepackage{amssymb}
\usepackage{enumitem}
\usepackage{ytableau}

\usepackage{hyperref}
\usepackage{color}
\usepackage{mwe}
\usepackage{float}

\usepackage{etoolbox}

\theoremstyle{plain}

\newtheorem{thm}{Theorem}[section]
\newtheorem{lem}[thm]{Lemma}
\newtheorem{prop}[thm]{Proposition}

\newtheorem{conj}[thm]{Conjecture}

\newtheorem{rmk}[thm]{Remark}

\newcommand\supp{\mathrm{supp}\;}
\newcommand\sgn{\mathrm{sgn}}

\newcommand\distinctsum{\sum_{\substack{j_1,...,j_n \in \mathbb{Z}^n\\ \mathrm{distinct}}}}
\newcommand\sS{\mathcal{S}}
\newcommand\RR{\mathbb{R}}

\definecolor {UMblue}  {RGB}{0, 39, 76}
\definecolor {UMmaize} {RGB}{255, 203, 5}

\definecolor {color_b}{RGB}{255,0,0}
\definecolor {color_c}{RGB}{20, 200, 30}
\definecolor {color_a}{RGB}{0,0,255}

\definecolor{lgreen} {RGB}{180,210,100}
\definecolor{dblue}  {RGB}{20,66,129}
\definecolor{ddblue} {RGB}{11,36,69}
\definecolor{lred}   {RGB}{220,0,0}
\definecolor{nred}   {RGB}{224,0,0}
\definecolor{norange}{RGB}{230,120,20}
\definecolor{nyellow}{RGB}{255,221,0}
\definecolor{ngreen} {RGB}{98,158,31}
\definecolor{dgreen} {RGB}{78,138,21}
\definecolor{nblue}  {RGB}{28,130,185}
\definecolor{jblue}  {RGB}{20,50,100}

\begin{document}

\title[The Alternative Hypothesis]{Higher Correlations and the Alternative Hypothesis}
\author{Jeffrey C. Lagarias}
\address{Department of Mathematics, University of Michigan, Ann Arbor, MI 48109-1043}
\email{lagarias@umich.edu}

\author{Brad Rodgers}
\address{Department of Mathematics and Statistics, Queen's University, Kingston, ON, Canada  K7L 3N6}
\email{brad.rodgers@queensu.ca}


\begin{abstract}
The Alternative Hypothesis concerns a hypothetical and unlikely picture of how zeros of the Riemann zeta function are spaced which one would like to rule out. 
In the Alternative Hypothesis, the renormalized distance between nontrivial zeros is supposed to always lie at a half integer. It is known that the Alternative Hypothesis
 is compatible with what is known about the pair correlation function of zeta zeros. We ask whether what is currently known about higher correlation functions of the zeros 
 is sufficient to rule out the Alternative Hypothesis and show by construction of an explicit counterexample point process that it is not. A similar result was recently independently obtained by T. Tao, using slightly different methods.

We also apply the ergodic theorem to this point process to show there exists a deterministic collection of points lying in $\tfrac{1}{2}\mathbb{Z}$ which satisfy the Alternative Hypothesis spacing but mimic the local statistics which are currently known about zeros of the zeta function.
\end{abstract}

\maketitle

\section{Introduction}
\label{sec:intro}

The purpose of this note is to address the following question: can the \emph{Alternative Hypothesis} regarding the distribution of zeros of the Riemann zeta function be ruled out by 
\emph{what is currently known about the local distribution of such zeros}? We explain more exactly what is meant by both these expressions below. The answer to this question is  ``{\em no}", 
and we will demonstrate this by  construction of a counterexample sequence $\{ c_j\}$ supported on $\frac{1}{2} \mathbb{Z}$, with spacing statistics mimicking what is known about the zeta zeros. 
A precise formulation  is given in Theorem \ref{alternative_GUE}.

Though the statement of Theorem \ref{alternative_GUE} requires no background from probability to understand, in our proof we find it natural to make use of the language of point processes. 
These are introduced in the next section, and our main result concerns a question in  {the realizability of point processes} (a topic recently studied by 
Kuna, Lebowitz, and Speer \cite{KuLeSp07,KuLeSp11}). In particular we construct an {\em Alternate Hypothesis Point Process}, with spacing statistics mimicking 
what is known about the zeta zeros and with the spacing between the points of any configuration always lying in the lattice $\frac{1}{2}\mathbb{Z}$. It is by sampling a 
configuration from this process that we obtain a deterministic sequence $\{c_j\}$ as above.

\section{Local Distribution of Zeta Zeros}
Before formulating  results  we review what we mean by the Alternative Hypothesis and the phrase `what is known about the local distribution of zeta zeros.'

\subsection{Spacings of zeros}

We label the non-trivial zeros of the Riemann zeta-function by $\{\sigma_j + i\gamma_j\}_{j \in \mathbb{Z}}$, ordered such that $\gamma_j \leq \gamma_{j+1}$. It dates back to 
Riemann that the ordinates $\gamma_j$ have density roughly $\tfrac{1}{2\pi}\log|T|$ near the value $T$, and for this reason is is natural to study the re-spaced ordinates
$$
\widetilde{\gamma}_j = \ \gamma_j \frac{\log |\gamma_j|}{2 \pi}.
$$
which we call {\em normalized zeros}. Thus respaced, the points $\widetilde{\gamma}_j$ have average density $1$; that is 
\begin{equation}
\label{1level}
\#\{\widetilde{\gamma}_j \in [0,T]\} \sim T
\end{equation}
as $T \rightarrow \infty$.

\subsection{GUE Hypothesis} 

Finer scale information regarding the spacing of zeros is predicted by the GUE Hypothesis.
Let
$$
S(x):= \begin{cases} \frac{\sin \pi x}{\pi x} & \textrm{for}\, x\neq 0 \\ 1 & \textrm{for}\, x=0. \end{cases}
$$
This function is a rescaled sinc-function, $S(x) = {\rm sinc} (\pi x)$ and $\int_{-\infty}^{\infty} S(x) dx=1$.

\begin{conj}[GUE Hypothesis]
\label{GUEHypothesis}
For $n\geq 1$ and any $\varphi \in \mathcal{S}(\mathbb{R}^n),$
$$
\lim_{T \to \infty} \frac{1}{T } \int_T^{2T} \distinctsum \varphi(\widetilde{\gamma}_{j_1}-t,...,\widetilde{\gamma}_{j_n}-t) \, dt = \int_{\mathbb{R}^n} 
\varphi(x)  \det_{1 \leq i \leq j \leq n} \big[ S(x_i - x_j) \big]\, d^n x.
$$
\end{conj}

Here $\mathcal{S}(\mathbb{R}^n)$ denotes the space of Schwartz class functions defined on $\mathbb{R}^n$.
By convention for $n=1$ the right side of the formula is $\int_{\mathbb{R}} \varphi(x_1) d x_1.$ 

The GUE Hypothesis emerged from work of Montgomery \cite{Mo73} and is by now widely believed to be true -- see e.g. Odlyzko \cite{Od87} for numerical evidence 
in its favor, and Katz and Sarnak  \cite{KaSa99}, \cite{KaSa99b}  for a theoretical discussion of analogous problems. For some test functions $\varphi$ 
Conjecture \ref{GUEHypothesis} is known to be true, and we discuss this below.

The case $n=2$ of the GUE Hypothesis already implies (see \cite{Mo73}) that the normalized distance between consecutive zeros becomes arbitrarily small and moreover that for any $\epsilon > 0$,
$$
\widetilde{\gamma}_{j+1} - \widetilde{\gamma}_j < \epsilon\quad \textrm{for a positive proportion of }j.
$$
A demonstration even that there exists some value $0< \delta < 1/2$ such that 
\begin{equation}
\label{gamma_often_small}
\widetilde{\gamma}_{j+1} - \widetilde{\gamma}_j \leq \delta \quad \textrm{for a positive proportion of }j
\end{equation}
would be of extraordinary interest since it would have as an implication that there are no Landau-Siegel zeros (see \cite{CoIw02}). For this reason it is of interest to explore models of zeros in which \eqref{gamma_often_small} is not true.\footnote{Recent work of Carniero et. al. \cite{CaChLiMi17} shows that $\widetilde{\gamma}_{j+1} - \widetilde{\gamma}_j \leq .606894$ for a positive proportion of $j$, which seems to be the best result known in this direction. See also \cite{So96}. Goldston and Turnage-Butterbaugh \cite{GoTuBu19} shows that $\liminf \; (\widetilde{\gamma}_{j+1} - \widetilde{\gamma}_j) \leq 0.50412$, using a result of Conrey, Ghosh, and Gonek \cite{CoGhGo84}. (The proof does not show that gaps become this small a positive proportion of the time however.) An unpublished result of Goldston and Milinovich shows that the method of \cite{CoGhGo84} cannot be used to show that $\liminf \; (\widetilde{\gamma}_{j+1} - \widetilde{\gamma}_j) \leq \mu$ if $\mu \leq 0.500026$.}

\subsection{The Alternative Hypothesis} An extreme example of such a model is described by Alternative Hypothesis, 
which says that almost all spacings 
of rescaled zeros $\widetilde{\gamma}_j- \widetilde{\gamma}_k$  lie close to half-integers.

\begin{conj}[The Alternative Hypothesis]
\label{AltHypothesis}
For all sufficiently large $j$,
$$
\widetilde{\gamma}_{j+1} - \widetilde{\gamma}_j = h_j + o(1),
$$
with $h_j$ an element of the set $\{\tfrac{1}{2}, 1, \tfrac{3}{2}, 2,...\}$ for all $j$.
\end{conj}

It is 
perhaps misleading to label this a conjecture, since almost certainly it is false!
Clearly it is contradicted by the GUE Hypothesis.  
Nonetheless, our knowledge about the Riemann zeta-function
 is sufficiently limited that no one has been able to disprove Conjecture \ref{AltHypothesis}. 
 See Farmer, Gonek and Lee \cite[Section 2]{FarGonLee14} and 
 Baluyot \cite{Ba16} for  recent papers drawing out potential implications of the Alternative Hypothesis and a discussion of background material. 

One memorable implication of the Alternative Hypothesis concerns a statistic that is  commonly called the (two-point) \emph{form factor}
 $K$ of the zeros (see \cite[Sect. 2]{FarGonLee14}, \cite[p. 82]{Meh04}); 
the Alternative Hypothesis can be shown to imply
\begin{equation}
\label{alt_formfactor}
\frac{1}{T} \sum_{0 \leq \widetilde{\gamma}_j, \widetilde{\gamma}_k \leq T} f(\widetilde{\gamma}_j - \widetilde{\gamma}_k) \sim \int_\mathbb{R} \hat{f}(\xi) K_{Alt}(\xi)\, d\xi.
\end{equation}
for\footnote{Allowing test functions $f(x) \in \sS(\mathbb{R})$, the form factor $K_{Alt}$  is a tempered distribution, where  $\delta_0(\xi)$ denotes a Dirac Delta function centered at $\xi=0$.}  $K_{Alt}(\xi):=\delta_{0}(\xi)+ |\xi|$ for $\xi\in [-1,1]$,  and defined by periodicity elsewhere. 

By contrast the case $n=2$ of the GUE Hypothesis implies \eqref{alt_formfactor} with $K_{Alt}$ replaced by $K_{GUE}(\xi):= \delta(\xi)+ \min(|\xi|,1)$. (See Cor. 1.4 of \cite{Ba16}, though this result is stated in a slightly different form there.)



\begin{figure}[H]
    \centering
    \begin{minipage}{0.46\textwidth}
        \centering
        \includegraphics[width=1\textwidth]{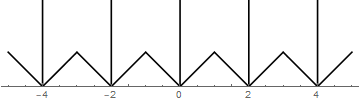} 
        \title{Form Factor $K_{{Alt}}(\xi)$}
    \end{minipage}\hfill
    \begin{minipage}{0.48\textwidth}
        \centering
        \includegraphics[width=1\textwidth]{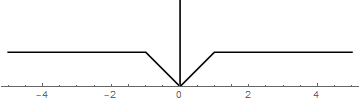} 
        \title{Form Factor $K_{GUE}(\xi)$}
    \end{minipage}
   \medskip 
	\caption{Form factors for the Alternative Hypothesis and GUE process (the vertical lines  represent Dirac Delta-functions)}
\label{fig:formfactor}
\end{figure}

\subsection{The local distribution of zeros: what is known}

In this paper our goal will be to show that, unlikely though it is, the Alternative Hypothesis is consistent with what is currently known about the statistical distribution of 
zeros. By \emph{what is known}, we mean the following partial version of the GUE Hypothesis that has been proved. We begin with what is known about the density of zeros.

\begin{thm}[GUE Density] 
\label{thm_correlations1}
 For any $\varphi \in \mathcal{S}(\mathbb{R})$, as $T \to \infty$, 
\begin{equation}
\label{thm_density}
\lim_{T \to \infty}\,  \frac{1}{T} \int_T^{2T} \sum_j \varphi(\widetilde{\gamma}_j-t)\, dt = \int_{\mathbb{R}} \varphi(x)\, dx.
\end{equation}
\end{thm}

The relation \eqref{thm_density}  is unconditional, and is easily derived from \eqref{1level}.
This density in effect dates back to  Riemann's 1859 memoir \cite{Ri1859} and  \eqref{1level}
was proved by von Mangoldt (see \cite[Ch. 14]{MoVa07}).

For higher-order correlations the GUE hypothesis has been verified for a band-limited\footnote{By a bandlimited function on $\mathbb{R}^n$ we mean a function that has a compactly supported Fourier transform.} collection of test functions.
\begin{thm}[Band-limited correlations] 
\label{thm_correlations}
Assume the Riemann Hypothesis. Let $\mathcal{K}_n$ denote the collection of 
band-limited test functions $\eta \in \mathcal{S}(\mathbb{R}^n)$ whose Fourier transform has support 
$$
\supp \hat{\eta} \subset \big\{ \xi \in \mathbb{R}^n:\, \xi_1+\cdots+\xi_n \neq 0 \quad \mathrm{or} \quad |\xi_1|+\cdots+|\xi_n| < 2\,\big\}.
$$
Then for all $n \ge 1$ for any $\eta \in \mathcal{K}_n$, 
\begin{equation}
\label{eta_correlations}
\lim_{T \to \infty} \, \frac{1}{T} \int_T^{2T} \distinctsum \eta(\widetilde{\gamma}_{j_1}-t,...,\widetilde{\gamma}_{j_n}-t)\, dt 
= \int_{\mathbb{R}^n} \eta(x)  \det_{1\leq i,j \leq n}\Big[ S(x_i - x_j) \Big]\, d^n x.
\end{equation}
\end{thm}

That \eqref{eta_correlations} is true for all $\eta$ with 
$$
\supp \hat{\eta} \subset \big\{ \xi \in \mathbb{R}^n:\, \xi_1+\cdots+\xi_n \neq 0 \big\}
$$
does not say anything special about the points $\{\widetilde{\gamma}_j\}$. 
It is true owing to the averaging in $T$; for test functions with Fourier transform supported in this domain both the left and right hand sides evaluate to $0$. 
However the  claim for band-limited functions $\eta$ having Fourier transform supported in the second region, 
$$
\supp \hat{\eta} \subset\{ \xi \in \mathbb{R}^n:\, |\xi_1|+\cdots+|\xi_n| < 2\},
$$
is much less trivial, at least for any 
part of this region intersecting $\xi_1+\cdots+\xi_n = 0$. The case $n=2$ is due to Montgomery \cite{Mo73}, $n=3$ is due to Hejhal \cite{He94}, and $n > 3$ is due to Rudnick and Sarnak \cite{RuSa96}.  Likely using sieve theoretic methods the region of $\xi$ such that $|\xi_1| + \cdots + |\xi_n| < 2$ can be extended to the closed region $|\xi_1| + \cdots + |\xi_n| \leq 2$, but this does not seem to be in the literature.

Note that  $\mathcal{K}_1 = \mathcal{S}(\mathbb{R})$, which corresponds to
Theorem \ref{thm_correlations1}. So no band-limitation is put on the test functions for $n=1$. We have singled out this case in Theorem \ref{thm_correlations1} to emphasize that in this case we have information for all test functions, as opposed to only partial information for $n \geq 2$.

\section{Main Result}
Our purpose in this note is to show that all information in Theorem \ref{thm_correlations} of Montgomery - Hejhal - Rudnick-Sarnak remains consistent with the Alternative Hypothesis. 

This matter did not seem to have been taken up previously in the literature, and its resolution depends upon somewhat special properties of the sine-determinant. Independently and very recently T. Tao \cite{Ta19} has also investigated the same question, and come to much the same conclusions, though by slightly different methods; Tao's starting point is the finite circular unitary ensemble and makes special use of symmetric function theory in order to do certain computations, whereas the method we give below begins with an infinite point process and makes special use of determinantal combinatorics and Poisson summation.

Our main result is as follows:

\begin{thm}
\label{alternative_GUE}
There exists a (deterministic) sequence of points $\{c_j\}_{j\in \mathbb{Z}}$ on the real line such that
$$
c_{j+1}-c_j \in \{\tfrac12, 1, \tfrac32, 2, ...\}
$$
for all $j$, and moreover for any $\varphi \in \mathcal{S}(\mathbb{R})$, 
$$
\lim_{T \to \infty} \, \frac{1}{T} \int_T^{2T} \sum_{j \in \mathbb{Z}} \varphi(c_j-t)\, dt  =\int_{\mathbb{R}} \varphi(x)\, dx,
$$
and for each $n\ge2$ and  any $\eta(x)  \in \mathcal{K}_n$, 
$$
\lim_{T \to \infty} \, \frac{1}{T} \int_T^{2T} \distinctsum \eta(c_{j_1}-t,...,c_{j_n}-t)\, dt  = \int_{\mathbb{R}^n}
\eta(x) \det_{1\leq i,j \leq n}\Big[ S(x_i - x_j) \Big]\, d^n x.
$$
\end{thm}

This theorem will be seen to follow from Theorem \ref{pp_counterexample},
 which constructs a counterexample point process,
which we term the Alternative Hypothesis (AH) point process.
The demonstration that such a sequence $\{c_j\}_{j\in \mathbb{Z}}$ exists is probabilistic,
using an ergodic property of the AH point process. It would be interesting if there is a `nice' explicit formula for the $j$th entry of some specific sequence satisfying this relation, but we do not pursue this here.

The construction we produce completely specifies  the  correlation functions for the sequence $\{c_j\}$ against all test functions in  the Schwarz spaces $\mathcal{S}(\mathbb{R}^n)$, 
not just those from $\mathcal{K}_n$ (see Remark \ref{AH_nlevel_formula} for a formula for the $n$-point densities). 
In Appendix \ref{sec:Gaps} we calculate the probability distribution of gaps between points for the Alternative Hypothesis point process in a closed form.
These probabilities coincide with those of the  distribution of  gaps for the $\frac{1}{2}$-discrete sine  process of  Proposition \ref{discrete_sine_exists}.

\section{Point processes}
\label{sec:to_pointprocesses}

\subsection{Outline of approach}
Our approach to proving Theorem \ref{alternative_GUE} makes use of machinery from the theory of point processes. 
We use the theory of determinantal point processes to construct a point process supported on $\tfrac{1}{2}\mathbb{Z}$ but with correlation functions consistent with the sine-determinant in \eqref{eta_correlations} for band-limited test functions $\eta$ -- that we can do so depends upon some special properties specific to the determinant of the functions $S(x_i-x_j)$. We randomly translate our point process to obtain one that is translation invariant in $\mathbb{R}$, and finally (again making use of machinery from the theory of determinantal point processes) we will obtain a sequence $\{c_j\}$ as in the theorem by taking an instance of the point process at random and using the ergodic theorem to guarantee almost surely this sequence inherits correlation functions.

\subsection{Review  of point processes}
We recall the definition of a point process and their basic properties. A well-written modern introduction with the same conventions we follow can be found in \cite[Ch. 4]{AnGuZe10}. Other well-written and concrete introductions can be found in \cite[Ch. 1]{HoKrPeVi09}, \cite{Gr77}, or \cite[Chap.3]{Res87}.

Intuitively, a point process is a random element whose values are configurations of points in some set. More formally, let $\mathcal{X}$ be either\footnote{More generally $\mathcal{X}$ can be any locally compact Polish space.} $\mathbb{R}$ or $a\mathbb{Z} = \{...,-a,0,a,...\}$, either equipped with the usual topology. A \emph{configuration} $u$ in $\mathcal{X}$ is just another name for a sequence of elements of $\mathcal{X}$: $u = (u_j)_{j\in\mathbb{Z}}$ with $u_i \in \mathcal{X}$ for all $i \in \mathbb{Z}$, and we use the notation $\#_V(u) := \{i: u_i \in V\}$ to count the number of members of $u$ that lie in a set $V$. We let $\mathrm{Conf}(\mathcal{X})$ be the set of locally finite configurations
\begin{align*}
\mathrm{Conf}(\mathcal{X}) := \{u\; &\textrm{a configuration in}\; \mathcal{X}\; \textrm{such that}\\& \textrm{for any compact interval}\; K,\, \#_K(u) < \infty\}.
\end{align*}
We let $\mathfrak{M}$ be the smallest topology on $\mathrm{Conf}(\mathcal{X})$ to contain all cylindrical sets $C_n^V$, where
$$
C_n^V := \{u \in \mathrm{Conf}(\mathcal{X}):\; \#_V(u) = n\},
$$
where $n$ ranges over $0,1,2...$ and, in the case $\mathcal{X} = \mathbb{R}$, $V$ ranges over all open intervals. (If $\mathcal{X} = a\mathbb{Z}$, we can let $V$ range over all singleton sets $\{x\} \subset a\mathbb{Z}$.) It is known that $(\mathrm{Conf}(\mathcal{X}), \mathfrak{M})$ is a Polish space.

Let $\mathcal{B}(\mathfrak{M})$ be the Borel $\sigma$-algebra generated by $\mathfrak{M}$. A \emph{point process} is a random element $u$ taking values in $(\mathrm{Conf}(\mathcal{X}), \mathcal{B}(\mathfrak{M}))$. Our definition of $\mathcal{B}(\mathfrak{M})$ provides that the sets
$$
\{ u \in \mathrm{Conf}(\mathcal{X}):\, \#_{B_1}(u) = n_1,...,\#_{B_k}(u) = n_k\}
$$
are all measurable events, for any collection of Borel subsets $B_1,...,B_k$ of $\mathcal{X}$ and non-negative integers $n_1,...,n_k$.

Also measurable is the event
$$
\{u \in \mathrm{Conf}(\mathcal{X}): \exists x \in \mathcal{X}\; \textrm{such that}\; \#_{\{x\}}(u) \geq 1\}.
$$
This is the event -- said another way -- that not all points of the configuration are distinct. If this event has probability $0$, so that almost surely all points of the point process are distinct, the point process is said to be \emph{simple}.

A point process on $\mathbb{R}$ or $a\mathbb{Z}$ with configurations $(u_j)$ is said to be \emph{translation invariant} if for any constant $x$ in $\mathbb{R}$ or $a\mathbb{Z}$ respectively, the random configuration $(x+u_j)$ has the same distribution as the original random configuration $(u_j)$.

For each $n$ and $\varphi \in C_c(\mathbb{R}^n)$, the function
\begin{equation}
\label{correlation_rv}
\distinctsum \varphi(u_{j_1},...,u_{j_n})
\end{equation}
is a measurable mapping from $\mathrm{Conf}(\mathcal{X})$ to $\mathbb{C}$.

For a point process with configurations $(u_j)$, \eqref{correlation_rv} is a complex-valued random variable. If there exists a measurable function $\rho_n(x_1,...,x_n)$ such that
$$
\mathbb{E}\,\distinctsum \varphi(u_{j_1},...,u_{j_n}) = \int_{\mathbb{R}^n} \varphi(x) \rho_n(x)\, dx,
$$
for all $\varphi \in C_c(\mathbb{R}^n)$, then $\rho_n$ is said to be the $n$-point correlation function (or $n$-point joint intensity) of the process $(u_j)$. For a point process on $a\mathbb{Z}$ instead of $\mathbb{R}$, the same is said to be true, where the measure $d^nx$ is replaced by $a^n$ times the counting measure:
$$
\mathbb{E}\,\distinctsum \varphi(u_{j_1},...,u_{j_n}) = a^n\sum_{k \in (a\mathbb{Z})^n}  \varphi(k) \rho_n(k).
$$
In some rare cases, given a collection of correlation functions $\rho_1,\rho_2,...$ there may be more than one point process with these correlation functions. 
See the discussion in Section 1.2 of  \cite{HoKrPeVi09}, especially
Remark 1.2.4. for a sufficient condition for uniqueness.

\subsection{Construction of determinantal point processes}
There is a special class of correlation functions which will be particularly important to us. Correlation functions, and the process generating them, are said to be \emph{determinantal} if there exists a function $K(\cdot,\cdot):\, \mathcal{X} \times \mathcal{X}  \rightarrow \mathbb{C}$ such that $K \in L^2(D\times D)$ for any compact $D \subset \mathcal{X}$ and
\begin{equation}
\label{determinantal_def}
\rho_n(x_1,...,x_n) = \det_{1\leq i,j \leq n}\Big[ K(x_i,x_j)\Big].
\end{equation}
There is an extensive literature on such processes, which play a role in random matrix theory, cf. Soshnikov \cite{So00}.

It is known that for a determinantal point process the correlation functions uniquely define the point process.

\begin{lem}
\label{det_unique}
Given $K$ as above, there is at most one point process with correlation functions given by \eqref{determinantal_def} for all $n\geq 1$.
\end{lem}

\begin{proof}
This is shown in  Lemma 4.2.6 of \cite{HoKrPeVi09}.
\end{proof}

Given a collection of functions $\rho_1, \rho_2, ...$ it can be difficult to decide whether there exists a point process for which these are the correlation functions. In the case of determinantal correlation functions however an exact characterization was given by Macchi \cite{Ma75}. We will discuss this result in Appendix \ref{subsec:A1}. Macchi's result can be used to demonstrate the existence of point processes with the following correlation functions:

\begin{prop}[Sine process]
\label{sine_exists}
There exists a point process on $\mathbb{R}$ with configurations $(v_j)$ such that
\begin{equation}
\label{sine_correlations}
\mathbb{E} \distinctsum \varphi(v_{j_1},...,v_{j_k}) = \int_{\mathbb{R}^n} \varphi(x)\, \det_{1\leq i, j \leq n}\Big[ S(x_i-x_j)\Big]\, d^nx,
\end{equation}
for all $n \geq 1$ and all $\varphi \in C_c(\mathbb{R}^n)$. We refer to this point process as the \textbf{sine process}.
\end{prop}

The sine process  is very well-known and a proof of its existence appears in \cite[Sec. 4.2]{AnGuZe10}.
\begin{prop}[Discrete sine process]
\label{discrete_sine_exists}
For each $0 < a \leq 1$ there exists a point process supported on $a\mathbb{Z}$ with configurations $(w_j)$ such that
\begin{equation}
\label{discrete_sine_correlations}
\mathbb{E} \distinctsum \varphi(w_{j_1},...,w_{j_k}) = a^n \sum_{k \in (a\mathbb{Z})^n} \varphi(k)\,\det_{1\leq i, j \leq n}\Big[ S(k_i-k_j)\Big]
\end{equation}
for all $n \geq 1$ and all $\varphi \in C_c(\mathbb{R}^n)$. We refer to this point process as the a-\textbf{discrete sine process}.
\end{prop}

The discrete sine process is less well-known, but it has also been made use of in e.g. \cite{BoOkOl00, Wi72, Wi94}. 
A proof of the existence for the discrete sine process is similar to that for the sine process, but does not seem to be readily accessible in the literature;  we provide one in Appendix \ref{sec:A}. The bound  $0 \le a \le 1$ is sharp: a  process  supported on $a\mathbb{Z}$ with correlation functions  \eqref{discrete_sine_correlations} does not exist for $a >1$, see Remark \ref{rmk:non-existence}.

It is well known that the sine process is simple. A similar  result holds  for the discrete sine-process; it is important
for our results and we include a  proof here for convenience.

\begin{prop}
\label{prop:simple}
For any $0< a \leq 1$, the $a$-discrete sine process is simple.
\end{prop}

\begin{proof}
For an element $x \in a\mathbb{Z}$, consider \eqref{discrete_sine_correlations} for $n=2$. 
There exists\footnote{Take $\varphi(x_1, x_2) = \varphi_1(x_1-x_2)$ with a bump function $\varphi_1(x) \in C_c^\infty(\mathbb{R})$ supported inside $(-a, a)$ with $\varphi_1(0) =1$.}
a  function $\varphi(x_1, x_2) \in C_c(\RR^2)$ which when restricted to $(a\mathbb{Z})^2$ has  $\varphi(k_1,k_2) = \mathbf{1}_{(x,x)}(k_1,k_2)$. 
We have
$$
\mathbb{E} \sum_{\substack{j_1,j_2 \in \mathbb{Z}\\ \mathrm{distinct}}} \mathbf{1}_{(x,x)}(w_{j_1}, w_{j_2})= \mathbb{E}\, [\#_{\{x\}}(w) ][\#_{\{x\}}(w)-1].
$$
On the other hand, by   \eqref{discrete_sine_correlations} we have 
\begin{eqnarray*}
 \mathbb{E} \sum_{\substack{j_1,j_2 \in \mathbb{Z}\\ \mathrm{distinct}}} \mathbf{1}_{(x,x)}(w_{j_1}, w_{j_2}) 
&=& a^2 \sum_{ (k_i, k_j) \in (a\mathbb{Z})^2} \varphi(k_i, k_j) \det \begin{pmatrix} S(k_i- k_i) & S(k_i-k_j) \\ S(k_i- k_j) & S(k_j-k_j) \end{pmatrix}\\
&=& a^2 \sum_{ k_i \in a\mathbb{Z} } \varphi(k_i, k_i) \det \begin{pmatrix} S(k_i-k_i) & S(k_i-k_i) \\ S(k_i- k_i) & S(k_i-k_i) \end{pmatrix} = 0.
\end{eqnarray*}
Hence for any $x$ almost surely $\#_{\{x\}}(w) = 0\textrm{ or }1$, and so almost surely all points of the configuration $(w)$ are distinct.
\end{proof}

It is  well-known that the sine process is 
invariant under translations in $\mathbb{R}$. We show that the  discrete sine-process is invariant under  translations in $a \mathbb{Z}$.

\begin{prop}
\label{prop:translation_invariant}
For any $0 < a \leq 1$, the $a$-discrete sine process is translation invariant, for translations in $a \mathbb{Z}$.
\end{prop}

\begin{proof}
The functions $\det_{1\leq i, j \leq n}[ S(k_i-k_j)]$ are invariant under real shifts  $(k_1,...,k_n) \mapsto (k_1+x,...,k_n+x)$.
It follows immediately that for any constant $x \in a\mathbb{Z}$ the $a$-discrete sine  process with configurations $(w_j)$ 
and a new process with configurations  $(w_j+x)$ has  the same correlation functions for all $n \ge 1$. Now Lemma \ref{det_unique}  implies that the 
new point process must also be the $a$-discrete sine process.
\end{proof}

We will need  to know the correlation function formula  \eqref{discrete_sine_correlations} holds for a slightly wider class of functions than $C_c((a\mathbb{Z})^n)$. In fact this formula is valid for all functions in $\ell^1\big((a\mathbb{Z})^n\big)$. Because the the discrete sine process is simple, for $\varphi \in \ell^1\big((a\mathbb{Z})^n\big)$, almost surely we have
$$
\distinctsum |\varphi(w_{j_1},...,w_{j_k})| \leq \sum_{\substack{k_1,...,k_n \in a\mathbb{Z} \\ \mathrm{distinct}}} |\varphi(k_1,...,k_n)| = ||\varphi ||_{1}. 
$$
The function $\det_{1\leq i, j \leq n}[S(x_i-x_j)]$ is bounded, so it is easy to see using monotone convergence that \eqref{discrete_sine_correlations} holds for $\varphi$ which are positive and summable. By decomposing general $\varphi$ into positive and negative parts, it follows that \eqref{discrete_sine_correlations} holds for all $\varphi \in \ell^1\big((a\mathbb{Z})^n\big)$. In particular, for our purposes,  \eqref{discrete_sine_correlations} holds for all $\varphi \in \mathcal{S}(\mathbb{R}^n)$. (The same  is true for the sine process for $\sS(\mathbb{R}^n)$.)

Note that the sine process has the same correlation functions as the GUE Hypothesis predicts for rescaled zeta zeros. In fact, the GUE Hypothesis can be restated in the following way: for each $T$ define a point process $\mathcal{X}_T$ with configurations\footnote{ $\mathcal{X}_T$ is not translation-invariant. We obtain translation-invariance only in the limit $T \to \infty$.} 
$$
(\widetilde{\gamma}_j - t)_{j\in\mathbb{Z}}
$$
with $t$ is chosen randomly and uniformly from $[T,2T]$. The GUE Hypothesis is the claim that these point processes $\mathcal{X}_T$ tend in distribution to the sine-kernel process as $T\rightarrow\infty$.

%
%
\subsection{Band-limited agreement}
We now demonstrate that the correlation functions of the discrete sine process agree with those of the sine process itself, when tested against a sufficiently band-limited class of test functions.

\begin{lem}
\label{sine_reconstruction_leq_1/2}
Let $(w_j)$ be the configurations of the discrete sine process on $a\mathbb{Z}$
with parameter $0< a \leq \frac{1}{2}$. Then for any $n\geq 1$ and any 
$\eta \in \mathcal{S}(\mathbb{R}^n)$ such that
\begin{equation}
\label{fourier_support}
\supp\; \hat{\eta} \subset [-\tfrac{1}{2a}, \tfrac{1}{2a}]^n,
\end{equation}
we have
\begin{equation}
\label{bandlimited_match}
\mathbb{E}\distinctsum \eta(w_{j_1},...,w_{j_n}) = \int_{\mathbb{R}^n} \eta(x) \det_{1 \leq i,j \leq n}\Big[ S(x_i-x_j)\Big]\, d^nx.
\end{equation} 
\end{lem}

\noindent \textbf{Remark:}  For Theorem \ref{alternative_GUE}, we need this result only for $a=1/2$.

\begin{proof}
By Proposition \ref{discrete_sine_exists}, verifying  \eqref{bandlimited_match} is equivalent to  verifying for  all $0 < a \le 1/2$ and $\eta \in \mathcal{S}(\mathbb{R}^n)$
with $\supp \; \hat{\eta} \subset [-\tfrac{1}{2a}, \tfrac{1}{2a}]^n$
 that
\begin{equation}
\label{sum_to_integral}
a^n \sum_{k \in (a\mathbb{Z})^n} \eta(k) \det_{1\leq i,j \leq n}\Big[ S(k_i-k_j)\Big] = \int_{\mathbb{R}^n} \eta(x) \det_{1\leq i, j \leq n}\Big[S(x_i-x_j)\Big]\, d^n x.
\end{equation}
The identity \eqref{sum_to_integral} says that
the $a$-discrete sine process and the sine process have perfectly matching $n$-point correlation statistics for all the  bandlimited test functions $\eta$
in the range specified above.

We prove  \eqref{sum_to_integral}   in two steps. First we show that the sine-determinant has no frequencies outside $[-1,1]^n$. 
Secondly this  implies that for band-limited $\eta$ the function $\eta(x)\det\big[S(x_i-x_j)\big]$ has no frequencies outside of $(-1-\tfrac{1}{2a}, 1 + \tfrac{1}{2a})^n$. 
This allows us to prove the result using Poisson summation.

\textbf{Step 1:} We show the sine-determinant has no frequencies outside $[-1, 1]^n$.
To do this we show that for any $g \in \mathcal{S}(\mathbb{R}^n)$ such that $\hat{g}(\xi) = 0$ for all $\xi \in [-1,1]^n$, we have
\begin{equation}
\label{bandlimit_sinedeterminant}
\int_{\mathbb{R}^n} g(x) \det_{1\leq i,j \leq n}\Big[S(x_i-x_j)\Big]\, d^n x = 0.
\end{equation}
Recall that
$$
S(x) = \int_\mathbb{R} \mathbf{1}_{[-\tfrac{1}{2}, \tfrac{1}{2}]} (\xi) e(x\xi) \, d\xi.
$$
In particular that $S(x)$ is bandlimited to $[-\frac{1}{2},\frac{1}{2}]$. We have for $n \ge 2$, 
\begin{align*}
\det_{1\leq i,j \leq n}\Big[ S(x_i-x_j)\Big] &= \sum_{\sigma \in S_n} \sgn(\sigma) \prod_{i=1}^n S(x_i - x_{\sigma(i)}) \\
&= \sum_{\sigma \in S_n} \sgn(\sigma)\int_{[-\tfrac{1}{2}, \tfrac{1}{2}]^n} e\Big( \sum_{i=1}^n ( x_i - x_{\sigma(i)})\xi_i\Big)\, d^n\xi \\
&= \sum_{\sigma \in S_n} \sgn(\sigma)\int_{[-\tfrac{1}{2} , \tfrac{1}{2}]^n} e\Big( \sum_{i=1}^n x_i( \xi_i - \xi_{\sigma^{-1}(i)})\Big)\, d^n\xi.
\end{align*}
Hence 
\begin{multline*}
\int_{\mathbb{R}^n} g(x) \det_{1\leq i,j \leq n}\Big[S(x_i-x_j)\Big]\, d^n x \\ = \sum_{\sigma \in S_n} \sgn(\sigma) \int_{[-\tfrac12, \tfrac12]^n} \hat{g}(\xi_1 - \xi_{\sigma^{-1}(1)}\,,...,\,\xi_n - \xi_{\sigma^{-1}(n)}) d^n\xi = 0,
\end{multline*}
the last equality holding because $|\xi_i - \xi_{\sigma^{-1}(i)}| \leq 1$ holds for  $1 \le i\le n$ over the full domain of integration. 
This verifies \eqref{bandlimit_sinedeterminant}.

\textbf{Step 2:} We let $f(x) = \eta(x) \det_{1 \leq i,j \leq n}\Big[ S(x_i-x_j)\Big]$ and we show first that
\begin{equation}
\label{f_is_bandlimited}
\hat{f}(y) = 0\quad \textrm{for}\quad y \notin (-1-\tfrac{1}{2a}, 1 + \tfrac{1}{2a})^n.
\end{equation}
To see this, let $y$ be as in \eqref{f_is_bandlimited}. We have
\begin{equation}\label{fourier_transform}
\hat{f}(y) = \int_{\mathbb{R}^n}  \eta(x) \det_{1 \leq i,j \leq n}\Big[S(x_i-x_j)\Big]\,e(-x\cdot y)\, d^n x,
\end{equation}
and $e(-x\cdot y) \eta(x)$ has Fourier transform 
$$
\hat{\eta}(\xi +y) = \int_{\mathbb{R}^n}  e(-x\cdot y) \eta(x) e(-x\cdot \xi) \,dx,
$$ 
which satisfies
$$
\hat{\eta}(\xi+y) = 0,  \quad \textrm{for} \quad  \xi \in [-1,1]^n,
$$
owing to the restricted support\footnote{For $a= 1/2$ one  must note that for $\eta \in \mathcal{S}(\mathbb{R}^n)$, the condition $\supp \hat{\eta} \subset [-1,1]^n$ implies that $\hat{\eta}(\xi) = 0$ for $\xi$ on the boundary of $[-1,1]^n$.} of $\hat{\eta}$ in \eqref{fourier_support}. Now step 1  verifies \eqref{f_is_bandlimited}.

Now note that Poisson summation implies
$$
a^n \sum_{k \in (a\mathbb{Z})^n} f(k) = \sum_{k \in (a^{-1}\mathbb{Z})^n} \hat{f}(k).
$$
Given $0 < a \leq 1/2$, one has that any nonzero $k \in (a^{-1}\mathbb{Z})^n$ lies outside $(-1-\tfrac{1}{2a},1+\tfrac{1}{2a})^n$, so that
$$
a^n \sum_{k \in (a^{1}\mathbb{Z})^n} f(k) = \hat{f}(0).
$$
Upon recalling the definition of $f$, and using \eqref{fourier_transform} this equality reduces  to the identity \eqref{sum_to_integral} that we set out to prove.
\end{proof}

The band-limited agreement between the sine-process and the $a$-discrete sine process ceases for $a > 1/2$. We return to this matter in section \ref{sec:more_questions}.

%
%
\subsection{Random translation: AH point process}
\label{subsec:AHPP}
Lemma \ref{sine_reconstruction_leq_1/2} is important for us in the critical case $a=1/2$. 
We can randomly translate the $\tfrac{1}{2}$-discrete sine process in order to induce a $\mathbb{R}$-translation invariant point process having: 
(i) the points in all configurations  separated by integer multiples of $1/2$ and (ii) the $n$-point correlation functions agree with those of the sine process against all band-limited test functions in $\mathcal{K}_n$ in Theorem \ref{thm_correlations}.

\begin{thm} [AH Point Process]
\label{pp_counterexample}
Let $(w_j)$ be the configurations of the $\tfrac{1}{2}$-discrete sine process on
$\frac{1}{2}\mathbb{Z}$ and let $\omega$ be a random variable chosen  uniformly from the interval $[0,1/2)$ independently of this process. Then
the new  point process 
on $\mathbb{R}$ having configurations
\begin{equation}\label{AH_pp}
(\tilde{w}_j) := (w_j-\omega)
\end{equation}
is simple,  $\mathbb{R}$-translation invariant,  and  satisfies, for all $n \ge 1$, 
\begin{equation}
\label{translate_pp}
\mathbb{E} \distinctsum \eta(\tilde{w}_{j_1},...,\tilde{w}_{j_n}) = \int_{\mathbb{R}^n} \eta(x_1,...,x_n) \det_{1\leq i, j \leq n}\Big[S(x_i-x_j)\Big]\, d^n x,
\end{equation}
for all $\eta \in \mathcal{K}_n$. Moreover any pair of points in any (realized) configuration $(\tilde{w}_j)$ will be separated by distances of $1/2, 1, 3/2,...$
\end{thm}

\begin{rmk}
\label{AH_nlevel_formula} 
{\em We call the process  \eqref{AH_pp} the {\it Alternative Hypothesis (AH) point process.}
 Note that, in full generality, the correlation functions of this process are as follows. For all $n \ge 1$ and all $\eta \in \mathcal{S}(\mathbb{R}^n)$,
\begin{multline}
\label{translate_pp_2}
\mathbb{E} \distinctsum \eta(\tilde{w}_{j_1},...,\tilde{w}_{j_n}) 
\\ = \frac{1}{1/2}\int_{0}^{1/2} (\frac{1}{2})^n \sum_{k \in (\frac{1}{2}\mathbb{Z})^n}   \eta(k_1-t,...,k_n-t) \det_{1\leq i,j \leq n}\Big[ S(k_i-k_j)]\Big]\, dt.
\end{multline}
This formula follows directly from \eqref{discrete_sine_correlations}.
}
\end{rmk}

\begin{proof}
We first address the final assertion on half-integer spacings. The distance of separation between the points of a configuration $(w_j-\omega)$ will be the same 
as that between the points of the $\tfrac{1}{2}$-discrete sine process, and so will be one of $0, 1/2, 1,...$. Because the discrete sine process is simple, we can remove $0$ from this list.

The simplicity of the new process is directly inherited from that of the $1/2$-discrete sine process.
The uniform averaging over $\omega \in [0,\frac{1}{2})$ converts the $\frac{1}{2} \mathbb{Z}$-translation invariance of the $1/2$-discrete sine process to $\RR$-translation invariance of
the new process.

On the other hand, verifying that \eqref{translate_pp} is satisfied is mainly a matter of book-keeping. We define subspaces
 $\mathcal{I}_n$ and $\mathcal{J}_n$ of $\mathcal{S}(\mathbb{R}^n)$ as follows:
\begin{enumerate}
\item[(i)]
 $\mathcal{I}_n$ is the collection of $\eta \in \mathcal{S}(\mathbb{R}^n)$ such that \\
$\supp \hat{\eta} \subset \{\xi \in \mathbb{R}^n:\, \xi_1+\cdots+\xi_n \neq 0\}$,
\item[(ii)]
 $\mathcal{J}_n$ is the collection of $\eta \in \mathcal{S}(\mathbb{R}^n)$ such that $\supp \hat{\eta} \subset [-1,1]^n.$
\end{enumerate}
Then we have 
$$
\mathcal{K}_n \subset \mathcal{I}_n + \mathcal{J}_n.
$$
To verify  \eqref{translate_pp} it suffices to  show it holds  for all $\eta \in \mathcal{I}_n$ and for all $\eta \in \mathcal{J}_n$.

\textbf{(i):} We consider $\mathcal{I}_n$ first. 
One may first verify that for $\eta \in \mathcal{I}_n$ and any $(x_1,...,x_n)\in \mathbb{R}^n$,
\begin{equation}
\label{oscillation_means_0}
\int_{-\infty}^\infty \eta(x_1-t,...,x_n-t)\,dt = 0.
\end{equation}
Now note that
\begin{multline}
\label{translated_correlations}
\mathbb{E} \distinctsum \eta(w_{j_1}-\omega,...,w_{j_n}-\omega) \\= \frac{1}{1/2} \int_0^{1/2} (1/2)^n \sum_{k\in (\tfrac12 \mathbb{Z})^n} \varphi(k_1-t,...,k_n-t) \det_{1\leq i,j \leq n}\Big[S(k_i-k_j)\Big]\, dt.
\end{multline}
Because the functions $\det_{1\leq i,j \leq n}\Big[S(k_i-k_j)\Big]$ are invariant under translations 
$(k_1,...,k_n) \mapsto (k_1-x,...,k_n-x)$, the right hand side of \eqref{translated_correlations} 
 reduces to
$$
\frac{1}{1/2} \int_{-\infty}^\infty (1/2 )^n \sum_{\substack{k\in (\tfrac12 \mathbb{Z})^n \\ k_1 = 0}} \varphi(k_1-t,...,k_n-t) \det_{1\leq i,j \leq n}\Big[S(k_i-k_j)\Big]\, dt.
$$
By dominated convergence (the functions $\eta$ have rapid decay) we can swap the order of integration and summation, and it follows from \eqref{oscillation_means_0} that for $\eta \in \mathcal{I}_n$ the expression above is $0$.

On the other hand, again because of the translation invariance of the sine-determinant, the right hand side of \eqref{translate_pp} is
$$
\int_{\mathbb{R}^{n-1}} \int_{-\infty}^\infty \eta(x_1-t,...,x_n-t) \det_{1\leq i, j \leq n}\Big[S(x_i-x_j)\Big]\, dt\, dx_2\cdots dx_n,
$$
where $x_1$ is no longer integrated over but instead we set $x_1=0$. And again for $\eta \in \mathcal{I}_n$ this expression is $0$. This establishes the claim for $\mathcal{I}_n$.

\textbf{(ii):} For $\eta \in \mathcal{J}_n$, the left hand side of \eqref{translate_pp} is
\begin{multline*}
\frac{1}{1/2}\int_0^{1/2}\mathbb{E} \distinctsum \eta(w_{j_1}-t,...,w_{j_n}-t)\, dt 
\\= \frac{1}{1/2}\int_0^{1/2}\int_{\mathbb{R}^n} \eta(x_1-t,...,x_n-t) \det_{1\leq i,j, \leq n}\Big[ S(x_i-x_j)\Big]\, d^n x \, dt\\
= \int_{\mathbb{R}^n} \eta(x) \det_{1\leq i,j \leq n}\Big[ S(x_i-x_j)\Big]\, d^n x \quad\quad\quad\quad
\end{multline*}
using Lemma \ref{sine_reconstruction_leq_1/2} in the first line and the translation invariance of the sine-determinant in the second.
\end{proof}

Theorem \ref{alternative_GUE} is a variant of 
Theorem  \ref{pp_counterexample}   in which the point process is replaced by a single sequence $\{c_j\}$. We will prove Theorem \ref{alternative_GUE} in the next section using the ergodic theorem.

%
%
\section{An ergodic construction}
\label{sec:construction}

\subsection{The ergodic theorem}
We review some basic 
definitions and
results from ergodic theory; see e.g. \cite[Ch. 28]{FrGr13} for a more extensive discussion. Let $\Omega$ be a measurable space. Recall that a sequence of random variables $X_0,X_1,...$ with each $X_\ell$ taking values in $\Omega$ is said to be \emph{stationary} if the random processes
$$
(X_\ell)_{\ell \in \mathbb{N}}\quad \mathrm{and}\quad (X_{\ell+1})_{\ell \in \mathbb{N}},
$$
have the same distribution (for $\mathbb{N} := \{\ell \geq 0\}$). A stationary sequence of random variables is said to be \emph{ergodic} if for any set $A \subset \Omega^{\mathbb{N}}$, with $A$ measurable with respect to the product $\sigma$-algebra, the event
$$
(X_{\ell+1})_{\ell \in \mathbb{N}} \in A
$$
is the same event as
$$
(X_{\ell})_{\ell \in \mathbb{N}} \in A,
$$
only when the latter event has probability $0$ or $1$. (Said another way, the shift-invariant $\sigma$-algebra for $(X_\ell)_{\ell \in \mathbb{N}}$ is trivial.)

\begin{thm}[Birkhoff's ergodic theorem]
If $(X_\ell)_{\ell \in \mathbb{N}}$ is an ergodic sequence of random variables, for $f$ a measurable complex-valued function with $\mathbb{E}|f(X_0)| < +\infty$,
$$
\lim_{N\rightarrow\infty} \frac{1}{N} \sum_{\ell=0}^{N-1} f(X_\ell) = \mathbb{E} f(X_0),
$$
almost surely.
\end{thm}

See for instance \cite[Thm. 1, Sec. 28.4]{FrGr13} for a proof of the ergodic theorem.
 
We will make use of this result when the random variable $X_0$ is a point processes, $X_1$ is a point process, etc.

\subsection{Proof of Theorem \ref{alternative_GUE}}

We need the following fact.

\begin{lem}
\label{sine_ergodic}
Consider the $\tfrac{1}{2}$-discrete sine process with configurations $(w_j)$. Let $X_0$ be the random configuration $(w_j)$, $X_1$ the random configuration $(w_j-1/2)$, $X_2$
 the random configuration $(w_j-1)$, etc. Then $X_0, X_1,...$ is an ergodic sequence.
\end{lem}

We prove this lemma in the  Appendix, in Section  \ref{subsec:ergodicity},  by reference to a general result (Theorem \ref{thm:general_ergodic}).
 By combining it with the Birkhoff ergodic theorem, we obtain:

\begin{thm}
\label{almostsure_correlations}
For almost all configurations $(w_j)$ of the $\tfrac{1}{2}$-discrete sine process, for all $n\geq 1$ and all $\varphi \in \mathcal{S}(\mathbb{R}^n)$,
\begin{multline}
\label{averaging_limit}
\lim_{N\rightarrow\infty} \frac{1}{N} \sum_{\ell=0}^{N-1} \bigg( \distinctsum \varphi(w_{j_1}-\tfrac{\ell}{2},..., w_{j_n} - \tfrac{\ell}{2})\bigg) \\
= (1/2)^n \sum_{k\in (\tfrac12 \mathbb{Z})^n} \varphi(k) \det_{1\leq i,j \leq n}\Big[ S(k_i-k_j)\Big].
\end{multline}
\end{thm}

\begin{proof}
For a single fixed $\varphi \in \mathcal{S}(\mathbb{R}^n)$, this follows directly from the ergodic theorem. To prove the result for all $n$ and $\varphi$ we make use of a density argument. Because the discrete sine process is simple, almost surely the points of $(w_j)$ are distinct, and in this case
\begin{equation}
\label{simple_upperbound}
\distinctsum |\varphi(w_{j_1}-\tfrac{\ell}{2},..., w_{j_n} - \tfrac{\ell}{2})| \leq \sum_{\substack{v_1, ..., v_n \in \tfrac12 \mathbb{Z} \\ \mathrm{distinct}}} |\varphi(v_1-\tfrac{\ell}{2},...,v_n-\tfrac{\ell}{2})| \leq \| \varphi \|_{\ell^1\big((\tfrac12 \mathbb{Z})^n\big)}.
\end{equation}
But $\ell^1\big((\tfrac12 \mathbb{Z})^n\big)$ is separable for each $n$. Take a countable dense subset $\{ \varphi_1^{(n)}, \varphi_2^{(n)},...\}$ of $\ell^1\big((\tfrac12 \mathbb{Z})^n\big)$. For each $n$ almost surely \eqref{averaging_limit} will be satisfied for all $\varphi_i^{(n)}$, since the union of exceptional configurations still has measure $0$, density and 
\eqref{simple_upperbound} then imply that \eqref{averaging_limit} will be satisfied for all $\varphi \in \mathcal{S}(\mathbb{R}^n)$ (in fact all $\varphi \in \ell^1[(\tfrac12 \mathbb{Z})^n]$). Moreover the union of all $n$ of 
the exceptional configurations must also have measure $0$, so we have proved the lemma.
\end{proof}

\begin{proof}[Proof of Theorem \ref{alternative_GUE}]
With probability $1$, configurations $(w_j)$ of the $\tfrac12$-discrete sine process consist of distinct points and are such that \eqref{averaging_limit} holds for all $n\geq 1$ and $\varphi \in \mathcal{S}(\mathbb{R}^n)$. Let $(c_j)$ be one such configuration. Clearly $c_{j+1}-c_j \in \{1/2, 1, 3/2,...\}$. 

For any $t$ for such a configuration,
\begin{multline*}
\lim_{N\rightarrow\infty} \frac{1}{N} \sum_{\ell=0}^{N-1} \bigg( \distinctsum \varphi(c_{j_1}-t - \tfrac{\ell}{2},..., c_{j_n}-t-\tfrac{\ell}{2})\bigg) \\ = (1/2)^n \sum_{k \in (\tfrac{1}{2}\mathbb{Z})^n} \varphi(k_1-t,...,k_n-t) \det_{1\leq i,j \leq n}\Big[S(k_i-k_j)\Big].
\end{multline*}
By \eqref{simple_upperbound}, for all $\ell$ and all $t \in [0,1/2)$,
\begin{equation}
\label{bounded_all_t}
\distinctsum  \varphi(c_{j_1}-t - \tfrac{\ell}{2},..., c_{j_n}-t-\tfrac{\ell}{2}) = O_\varphi(1),
\end{equation}
so averaging $t$ from $0$ to $1/2$ and using dominated convergence to interchange the order of averaging and limit,
\begin{multline*}
\lim_{N\rightarrow\infty} \frac{1}{N} \frac{1}{1/2}\int_0^{1/2}\sum_{\ell=0}^{N-1} \bigg( \distinctsum \varphi(c_{j_1}-t - \tfrac{\ell}{2},..., c_{j_n}-t-\tfrac{\ell}{2})\bigg)\, dt \\ = \frac{1}{1/2} \int_0^{1/2} (1/2)^n \sum_{k \in (\tfrac{1}{2}\mathbb{Z})^n} \varphi(k_1-t,...,k_n-t) \det_{1\leq i,j \leq n}\Big[S(k_i-k_j)\Big]\, dt.
\end{multline*}
The left hand side above is
\begin{equation}
\lim_{N\rightarrow\infty} \frac{1}{N/2} \int_0^{N/2} \distinctsum \varphi(w_{j_1}-t,...,w_{j_n}-t)\, dt,
\end{equation}
while the right hand side is the $n$-point correlation function of the process $(w_j - \omega)$ in Theorem \ref{pp_counterexample} (recall \eqref{translated_correlations}), and therefore agrees for test functions in $\mathcal{K}_n$ with the correlation functions of the sine process itself.

It remains to verify
\begin{multline}
\label{discrete_to_dyadic_continuous}
\lim_{T\rightarrow\infty} \frac{1}{T} \int_T^{2T} \distinctsum \varphi(c_{j_1}-t,...,c_{j_n}-t)\,dt \\= \lim_{N\rightarrow\infty} \frac{1}{N/2} \int_0^{N/2} \distinctsum \varphi(c_{j_1}-t,...,c_{j_n}-t)\,dt.
\end{multline}
This is done in the following way: the right hand side of \eqref{discrete_to_dyadic_continuous}, with a limit over integers $N$, may be seen to be equal to a continuous limit
$$ 
\lim_{T\rightarrow\infty} \frac{1}{T} \int_0^{T} \distinctsum \varphi(c_{j_1}-t,...,c_{j_n}-t)\,dt,
$$
using the pointwise bound \eqref{bounded_all_t}. This may in turn be seen to be equal to the left hand side of \eqref{discrete_to_dyadic_continuous} by subtracting an integral from $0$ to $T$ from an integral from $0$ to $2T$. We omit the details.
\end{proof}

%
%
\section{Concluding Remarks}
 \label{sec:more_questions}

The underlying construction of this paper showed that there
exists a point process supported on the lattice $\frac{1}{2} \mathbb{Z}$
that  had $n$-point correlation statistics perfectly matching that of the sine process for  a class of bandlimited test functions on $[-1,1]^n$, for all $n \ge 1$. 
More generally one may consider point supported on a lattice $a \mathbb{Z}$ and ask for what $B$ can a perfect match be achieved for all test functions bandlimited to an interval $[-B, B]^n$, for all $n \ge 1$. 
For the sine process in Lemma \ref{sine_reconstruction_leq_1/2}
we gave an existence region where it can be done, achieving $B= \frac{1}{2a}$ for $0 < a \le \frac{1}{2}.$ This sort of problem can be considered for more general $a$ and $B$ and for more general point processes. We address such questions in a  companion paper \cite{LaRo19}.

Another question concerns uniqueness: is the AH point process in Theorem \ref{pp_counterexample} the unique point process having (i) the points in all configurations separated by integer multiples of $1/2$ and (ii) the $n$-point correlation functions agreeing with those of the sine-process against all band-limited test functions in $\mathcal{K}_n$ as in Theorem \ref{thm_correlations}? We  expect other point processes with these properties exist.  An analogous question was  previously posed less formally in \cite[Question A.5]{AIM}. 

  Recently Farmer, Gonek and Lee  \cite{FarGonLee14} made
  a   study of    the (normalized) spacings of the
 zeros of the derivative Riemann $\xi'(s)$ of the Riemann $\xi$-function $\xi(s) = \frac{1}{2} s(s-1)\pi^{-\frac{s}{2}} \Gamma( \frac{s}{2})$,
 as a possible approach to disprove the Alternative Hypothesis, see \cite[Section 2]{FarGonLee14}.
 Assuming the Riemann hypothesis, all zeros of $\xi'(s)$ lie on the critical line. Assuming RH, they  proved bounds on the zero spacings,
 that $\liminf_{n \to \infty}(\gamma_n^{'} - \gamma_n^{'}) \frac{\log \gamma_n^{'}}{2\pi}  \le 0.897$, where
 $\gamma_n^{'}$ denote the ordinates of the zeros of  $\xi'(s)$ in the upper half plane, ordered in increasing order.   
   Sodin \cite{Sod17} presented a family of point process $\mathfrak{Si}_a$ depending on a real parameter $a$ modeling 
 critical points of characteristic polynomials of  GUE-random matrices.
 He showed under very strong assumptions (the Riemann Hypothesis and the GUE Hypothesis\footnote{In Corollary 2.3 Sodin uses the term `multiple correlation conjecture', meaning convergence
 of (shifted) normalized zeta zeros to the sine process in distribution. This  statement is equivalent to the GUE Hypothesis.})
  that the process $\mathfrak{Si}_0$ with $u=0$ would model the $ T \to \infty$ limit of 
  (shifted) normalized spacings of zeros of $\xi'(s)$ (\cite[Corollary 2.3]{Sod17}). 
  The point process  $\mathfrak{Si}_0$ has point spacings more tightly clustered about the mean spacing  $1$ than those of the sine process. 
   One may ask a parallel question: what is the  range of bandlimited agreement permissible
   for correlation functions of a point 
   process supported on a lattice $a\mathbb{Z}$ to that of the point process  $\mathfrak{Si}_0$.

\bigskip


\noindent {\bf Acknowledgments.}
Work of the  first author was partially supported by  
NSF grant DMS-1701576, a Chern Professorship at MSRI in Fall 2018 and by a  Simons Foundation Fellowship in 2019. MSRI is partially supported by an NSF grant. Work of the second author was partially supported by  NSF grant DMS-1701577 and by an NSERC grant. We thank F. Aryan for questions that led to Appendix \ref{sec:Gaps}. We thank S. Baluyot, D. Goldston, D. Fiorilli, S. Sodin, and T. Tao for other discussions related to this paper, and an anonymous referee for helpful comments and corrections.

\appendix

\section{The discrete sine process: existence and ergodicity}\label{sec:A}

In this appendix we show that the discrete sine process exists, as claimed by Proposition \ref{discrete_sine_exists} and demonstrate that it is ergodic, as claimed by Lemma \ref{sine_ergodic}.

\subsection{Existence}
\label{subsec:A1}

In what follows\footnote{Macchi's theorem holds more generally for $\mathcal{X}$ a locally compact Polish space.} 
$\mathcal{X} = a\mathbb{Z}$, and we give $\ell^2(\mathcal{X})$ the inner product $\langle f, g \rangle := a \sum_{k \in a\mathbb{Z}} f(k)\overline{g(k)}$. A general criterion for the realizability\footnote{ Macchi \cite{Ma75},   Sec. IV. B. Definition 7, treats determinantal point processes (called fermion processes).} of point processes is as follows (Macchi \cite[Theorem 12]{Ma75}):

\begin{thm}[Macchi]
\label{Macchi}
On the space $\mathcal{X} = a\mathbb{Z}$
define the operator $\mathbb{K}:\, \ell^2(\mathcal{X}) \rightarrow \ell^2(\mathcal{X})$ by
$$
\mathbb{K} f(j) = a\sum_{k \in a\mathbb{Z}} K(j,k) f(k).
$$
If $\mathbb{K}$ is a Hermitian locally trace class\footnote{An operator $\mathbb{K}:\, L^2(E,d\mu) \rightarrow L^2(E,d\mu)$ is said to be \emph{locally} trace class if for any compact $D \subset E$ the operator $\mathbb{K}_D:\, L^2(D,d\mu) \rightarrow L^2(D,d\mu)$ defined by
$$
\mathbb{K}_D f(x) = \int_D K(x,y) f(y)\, d\mu(y)
$$
is trace class.} 
operator, then the functions 
$$
\rho_n(k):= \det_{1 \leq i,j \leq n}\Big[ K(k_i,k_j)\Big]
$$
are the correlation functions of a point process on $\mathcal{X}$ if any only if\footnote{We make use of the usual notation for operators, so that $0 \leq \mathbb{K} \leq I$ means that for any $\psi \in \ell^2(\mathcal{X})$ we have
$$
0 \leq \langle \psi, \mathbb{K} \psi \rangle \leq \langle \psi, \psi \rangle.
$$}
$$
0 \leq \mathbb{K} \leq I.
$$
\end{thm}

\begin{proof}
The theorem with this  statement is proved as Theorem 3 of Soshnikov \cite{So00}.
\end{proof}

We make use of this result for only one kernel $K$, and it is not very esoteric, but we  nonetheless have to verify that it is locally trace class. We will not need anything beyond the 
following result, a specialization of the result stated in Reed and Simon \cite[XI.4, p. 65]{ReSi79}.

\begin{lem}
\label{traceclass}
For a function $K:\, \mathcal{X} \times \mathcal{X} \rightarrow \mathbb{C}$, suppose the operator $\mathbb{K}$ defined as before is non-negative definite. If
$$
\sum_{j \in \mathcal{X}} K(j,j) < +\infty
$$
then $\mathbb{K}:\, \ell^2(\mathcal{X}) \rightarrow \ell^2(\mathcal{X})$ is trace class.
\end{lem}

With these tools we can prove Proposition \ref{discrete_sine_exists}.

\begin{proof}[Proof of Proposition \ref{discrete_sine_exists}]
For $0< a\leq 1$, we seek to show there is a point process on $\mathcal{X} = a \mathbb{Z}$ with correlation functions
$$
\det_{1\leq i,j \leq n}\big[S(k_i-k_j)\big].
$$ 
By Macchi's Theorem we 
need only show that the operator $\mathbb{K}:\, \ell^2(\mathcal{X}) \rightarrow \ell^2(\mathcal{X})$ defined by
$$
\mathbb{K}f(j) = a \sum_{k \in a\mathbb{Z}} S(j-k) f(k),
$$
is Hermitian, locally trace class, and satisfies $0 \leq \mathbb{K} \leq I$. Define the Fourier transform 
$\mathcal{F}:\, \ell^2(\mathcal{X}) \rightarrow L^2[-\tfrac{1}{2a},\tfrac{1}{2a}]$ by
$$
(\mathcal{F}f)(\theta) = a\sum_{k \in a\mathbb{Z}} f(k) e^{-2\pi i k \theta}, \quad \quad (\mathcal{F}^{-1}F)(j) = \int_{-1/2a}^{1/2a}  F(\theta) e^{ 2\pi i  j \theta}\, d\theta,
$$
Recall  that, as long as $0 < a \leq 1$, 
\begin{equation}\label{cutoff-sine-kernel}
S(x) = \int_{-1/(2a)}^{1/(2a)} \mathbf{1}_{[-\tfrac{1}{2}, \tfrac{1}{2}]}(\theta) e^{2 \pi i x \theta } \,  d \theta;
\end{equation}
the equality does not hold for $a>1$.
Using \eqref{cutoff-sine-kernel} we  see that
$$
\mathbb{K}f = \mathcal{F}^{-1} G, \quad \textrm{for}\quad G(\theta):= \mathbf{1}_{[-\tfrac{1}{2}, \tfrac{1}{2]}}(\theta) \cdot (\mathcal{F}f)(\theta).
$$
That is, $\mathbb{K}$ is a projection of $\ell^2(\mathcal{X})$ onto the space of functions $g \in \ell^2(\mathcal{X})$ with 
Fourier transform supported on $[-\tfrac{1}{2}, \tfrac{1}{2}]$. That $\mathbb{K}$ is a projection shows that it is Hermitian, that $0 \leq \mathbb{K} \leq I$, and also that $\mathbb{K}$ is 
non-negative definite. That $\mathbb{K}$ is locally trace class follows obviously and immediately from Lemma \ref{traceclass}.
\end{proof}

\begin{rmk}\label{rmk:non-existence}
{\em For $a > 1$, the operator $\mathbb{K}$ ceases to satisfy $0 \leq \mathbb{K} \leq 1$, and so by Macchi's theorem there does not exists an $a$-discrete sine process for $a >1$. Indeed, for $\psi(0):= 1$ and $\psi(k) := 0$ for all $k\neq 0$, we have
$$
\langle \psi, \mathbb{K}\psi \rangle = a^2,
$$
while
$$
\langle \psi, \psi \rangle = a.
$$
and $a^2>a$ for $a>1$.
}
\end{rmk} 
\subsection{Ergodicity}
\label{subsec:ergodicity}

To establish Lemma \ref{sine_ergodic} we make use of another general result.
\begin{thm}
\label{thm:general_ergodic}
Let $\mathcal{X} = a\mathbb{Z}$ and suppose $(u_j)$ are the configurations of a translation invariant determinantal point process on $\mathcal{X}$, with a kernel satisfying the conditions of Macchi's theorem and with $K(j,k)$ representable as $K(j-k)$ with $\lim_{|x|\rightarrow\infty} K(x) = 0$. Then the sequence of point processes with configurations
$$
(u_j),\; (u_j-a),\; (u_j-2a),\; ...
$$
is ergodic.
\end{thm}

\begin{proof} 
This is the content of Theorem 4.2.34 of \cite{AnGuZe10}. The result there is proved for $\mathcal{X} = \mathbb{R}^d$, but the proof carries over straightforwardly to $\mathcal{X} = a\mathbb{Z}$.
\end{proof}

Because $S(x) \rightarrow 0$ as $|x| \rightarrow\infty$, Theorem \ref{thm:general_ergodic} immediately implies Lemma \ref{sine_ergodic}.

%
%
\section{Gap Probabilities for the Alternative Hypothesis Point Process }\label{sec:Gaps}

This Appendix derives a formula for the probabilities that a random gap in the Alternative Hypothesis point process (defined in Sect. \ref{subsec:AHPP}) is of the size $1/2$, $1$, $3/2$, etc. Let $G_{L/2}$ be the conditional probability that  a configuration in the $1/2$-discrete sine-kernel process 
has position  $L/2$  occupied by a point while  none of positions $(L-1)/2, (L-2)/2,...,1/2$ are occupied, given that the position $0$ is occupied. That is, for the $1/2$-discrete sine-kernel process,
$$
G_{L/2} = \mathbb{P}[\#_{\{L/2\}}=1, \, \#_{\{1/2,...,(L-1)/2\}}=0\, \Big| \,  \#_{\{0\}} =1].
$$
Because the $1/2$-discrete sine-kernel process is translation-invariant, and because the alternative hypothesis point process is just the $1/2$-discrete sine-kernel process with 
an independent translation, $G_{L/2}$ is the probability that a randomly chosen gap from either of these processes is of size $L/2$.

In order to state the result, for $L \geq 1$, we let $\Sigma_L$ be the $L\times L$ symmetric Toeplitz matrix
\begin{equation*}
\label{eqn:Toeplitz}
\Sigma_L = (a_{k-j})_{0 \leq j,k \leq L-1},
\end{equation*}
where
$$
a_j := S(j/2) = \begin{cases} 1 & \textrm{for}\; j=0 \\ 0 & \textrm{for} j \neq 0 \, \textrm{even} \\ \frac{2(-1)^{(|j|-1)/2}}{\pi |j|} & \textrm{for} j \, \textrm{odd}, \end{cases}.
$$
For example, noting that $a_j = a_{-j}$, we have
$$
\Sigma_4 = 
\begin{pmatrix} 1 & a_1 & 0 & a_3 \\ a_{-1} & 1 & a_1 & 0 \\0 & a_{-1}  & 1 & a_1 \\a_{-3} & 0 & a_{-1} &1
\end{pmatrix}
=
\begin{pmatrix} 1 & \tfrac{2}{\pi} & 0 & -\tfrac{2}{3\pi} \\ \tfrac{2}{\pi} & 1 & \tfrac{2}{\pi} & 0 \\0 & \tfrac{2}{\pi}  & 1 & \tfrac{2}{\pi} \\\ -\tfrac{2}{3\pi} & 0 & \tfrac{2}{\pi} &1
\end{pmatrix}
$$
We  let
$$
\omega(L) = \det(I - \frac{1}{2}\Sigma_L),
$$
for $I$ the $L\times L$ identity matrix, with the convention that
$$
\omega(0) = 1.
$$
For example, 
$$
\omega (4) =  \det \Big(\frac{1}{2} \begin{pmatrix} 1 & -a_1 & 0 & -a_3 \\-a_1 & 1 & -a_1 &0 \\  0& -a_1  & 1 & -a_1\\ -a_3 & 0 & -a_1 &1 \end{pmatrix}\Big)
= \frac{1}{16} \det  \begin{pmatrix} 1 & -\tfrac{2}{\pi} & 0 & \tfrac{2}{3\pi} \\ -\tfrac{2}{\pi} & 1 & -\tfrac{2}{\pi} & 0 \\0 & -\tfrac{2}{\pi}  & 1 & -\tfrac{2}{\pi} \\ \tfrac{2}{3\pi} & 0 & -\tfrac{2}{\pi} &1
\end{pmatrix}.
$$


\begin{thm}
\label{thm:gaps}
For the Alternative Hypothesis point process the gap probabilities $G_{L/2}$ for $L \ge 1$ satisfy
$$
G_{L/2} = 2(\omega(L+1)+ \omega(L-1) - 2\omega(L)),
$$
where $\omega(L)= \det(I - \frac{1}{2}\Sigma_L)$.
\end{thm}

We record the following values:
\begin{align*}
G_{1/2} &= \tfrac{1}{2} - \tfrac{2}{\pi^2} \approx .297 \\
G_1 &= \tfrac{1}{4}+\tfrac{2}{\pi^2} \approx .453 \\
G_{3/2} &= \tfrac{1}{8} +  \tfrac{4}{9 \pi^2} + \tfrac{32}{9 \pi^4} \approx .207 \\
G_2 &= \tfrac{1}{16} + \tfrac{1}{18 \pi^2} - \tfrac{224}{81 \pi^4}  \approx 0.0397 \\
G_{5/2} &= \tfrac{1}{32} - \tfrac{209}{1800 \pi^2} - \tfrac{1664}{2025 \pi^4}
 - \tfrac{131072}{18225 \pi^6}   \approx 0.00357 \\
G_3 &= \tfrac{1}{64} - \tfrac{3}{ 25 \pi^2}- \tfrac{13312}{16875 \pi^4}
 + \tfrac{2097152}{455625 \pi^6}   \approx 0.000156
\end{align*}

We emphasize that these are gap probabilities for the Alternative Hypothesis point process; there may exist other models of zeros which are consistent with the alternative hypothesis but which have different gap probabilities. Farmer et al \cite{ FarGonLee14}
and Aryan \cite{Ary19} have considered upper and lower bounds for gap probabilities in this more general setting.

Theorem \ref{thm:gaps} is a consequence of a well-known tool from the theory of determinantal point processes.
In calculations below we  adopt the convention that for $x \in \mathbb{R}$,
\begin{equation}
\label{eq:0powers}
0^x = \begin{cases} 1 & \textrm{if}\; x = 0 \\ 0 & \textrm{otherwise}. \end{cases}
\end{equation}

\begin{thm}
\label{thm:gaps_to_det}
For all $z \in \mathbb{C}$, for a determinantal point process on $a\mathbb{Z}$ with kernel $K : (a\mathbb{Z}) \times (a\mathbb{Z}) \rightarrow \mathbb{R}$ as in \eqref{determinantal_def}, for any finite $B \subset a\mathbb{Z}$,
$$
\mathbb{E} z^{\#_B} = \det(I + (z-1) K|_B),
$$
where $K|_B \,: L^2(B)\rightarrow L^2(B)$ is the operator defined by
$$
(K|_B \phi)(j) = a \sum_{k \in B} K(j,k) \phi(k), \quad \textrm{for}\; j \in B.
$$
\end{thm}

\begin{proof}
This is a special case of Theorem 2 in \cite{So00}. This result interprets $z^{\#_B}$ for $z=0$ with the convention \eqref{eq:0powers}.
\end{proof}

We can now prove  Theorem \ref{thm:gaps}.

\begin{proof}[Proof of Theorem \ref{thm:gaps}]
We will specialize Theorem \ref{thm:gaps} to the $1/2$-discrete sine-kernel process, and take $z = 0$. Note that for $L \geq 2$,
\begin{multline*}
\mathbb{P}(\#_{\{L/2\}}=1, \#_{\{1/2,..., (L-1)/2\}} = 0 \,\Big|\, \#_{\{0\}} = 1) \\
= \frac{\mathbb{P}(\#_{\{L/2\}}=1, \#_{\{1/2,..., (L-1)/2\}} = 0, \#_{\{0\}} = 1) }{\mathbb{P}(\#_{\{0\}}) = 1)} \\
= \frac{\mathbb{E} ( (1- 0^{\#_{\{L/2\}}})\,0^{\#_{\{1/2,..., (L-1)/2\}}} \, (1 - 0^{\#_{\{0\}}} ))}{\mathbb{E} \#_{\{0\}}},
\end{multline*}
using the fact that this process is simple. But $\mathbb{E} \#_{\{0\}} = 1/2$, and
\begin{multline*}
\mathbb{E} ( (1- 0^{\#_{\{L/2\}}})\,0^{\#_{\{1/2,..., (L-1)/2\}}} \, (1 - 0^{\#_{\{0\}}} )) \\ = \mathbb{E} \, 0^{\#_{\{1/2,...,(L-1)/2\}}} + \mathbb{E}\, 0^{\#_{\{0,...,L/2\}}} - \mathbb{E}\, 0^{\#_{\{1/2,...,L/2\}}} - \mathbb{E} \,0^{\#_{\{0,...,(L-1)/2\}}}.
\end{multline*}
Using Theorem \ref{thm:gaps_to_det} and the fact that the $1/2$-discrete sine-kernel has the kernel $(x,y)\mapsto S(x-y)$, we see that
$$
\mathbb{E} \, 0^{\#_{\{1/2,...,(L-1)/2\}}} = \det(I - \frac{1}{2} \Sigma_{L-1}),
$$
$$
\mathbb{E}\, 0^{\#_{\{0,...,L/2\}}} = \det(I - \frac{1}{2} \Sigma_{L+1}),
$$
$$
\mathbb{E}\, 0^{\#_{\{1/2,...,L/2\}}} = \mathbb{E} \,0^{\#_{\{0,...,(L-1)/2\}}} = \det(I - \frac{1}{2} \Sigma_{L}).
$$
Substituting these values above proves the theorem for $L \geq 2$. For $L=1$, we have
$$
\mathbb{P}( \#_{\{1/2\}} = 1\, \Big| \, \#_{\{0\}} = 1) = \frac{\mathbb{E}\, (1 - 0^{\#_{\{1/2\}}}) (1 - 0^{\#_{\{0\}}})}{\mathbb{E} \#_{\{0\}} },
$$
and the proof of the claim proceeds the same way.
\end{proof}


\end{document}